\documentclass[11pt,a4paper]{article}
\usepackage[T1]{fontenc}
\usepackage[utf8]{inputenc}
\usepackage{lmodern}
\usepackage[english]{babel}
\usepackage[margin=1in]{geometry}
\usepackage{amsmath,amssymb,amsthm,mathtools,tikz-cd}
\usepackage{microtype,enumitem,needspace}
\usepackage[hidelinks]{hyperref}
\newtheorem{theorem}{Theorem}[section]
\newtheorem{proposition}[theorem]{Proposition}
\newtheorem{lemma}[theorem]{Lemma}
\newtheorem{corollary}[theorem]{Corollary}
\theoremstyle{definition}

\newtheorem{remark}[theorem]{Remark}
\newcommand{\Q}{\mathbb Q}\newcommand{\Z}{\mathbb Z}\newcommand{\F}{\mathbb F}
\newcommand{\A}{\mathbb A_f}\newcommand{\LL}{\widehat{\mathbb L}}
\newcommand{\limi}{\varprojlim}\newcommand{\limone}{\varprojlim\nolimits^{1}}

\DeclareMathOperator{\im}{im}\DeclareMathOperator{\coker}{coker}
\DeclareMathOperator{\Hom}{Hom}\DeclareMathOperator{\holim}{holim}
\DeclareMathOperator{\Lie}{Lie}
\DeclareMathOperator{\cdp}{cd_p}
\title{Adelic Asphericity and One-Relator Pro-\(p\) Groups}
\author{Andrey Mikhovich}\date{}
\hypersetup{pdfauthor={Andrey Mikhovich},pdftitle={Adelic Asphericity and One-Relator Pro-p Groups}}
\begin{document}
\maketitle
\begin{abstract}
We give a direct Fox-calculus proof of the known rational asphericity theorem for every finite presentation complex with one nontrivial relator, including proper powers. The proof uses completed Fox derivatives and the absence of zero divisors in the completed algebra of a prounipotent group. Exact scalar extension to the finite adeles then yields adelic asphericity and a product decomposition of the higher homotopy groups of integral Bousfield--Kan completion. We also prove rational and pro-\(p\) analogues of Mannan's formula for finite two-complexes, with their natural topologies and fundamental group actions. The rational cohomological Whitehead criterion is proved directly. For a discrete one-relator group with relator not a proper power, we show that the tower \(L_2\nu_n(G)\) is Mittag--Leffler. Consequently \(\limone_nL_2\nu_n(G)=0\), and \(\pi_2(\Z_\infty BG)\) is naturally the inverse limit of the nilpotent Baer invariants. We identify this Baer limit canonically and topologically with the kernel of the integral augmentation-completed Fox operator and with the product of its local pro-\(p\) kernels. The comparison yields a torsion-free profinite homotopy module, vanishing inverse limits of the finite-stage transgression kernels, and explicit lifting obstructions. These results provide a homotopy-theoretic framework for the question of Serre in Gildenhuys's corrected formulation.
\end{abstract}
\noindent\textbf{Keywords:} Bousfield--Kan completion, finite adeles, one-relator pro-\(p\) groups, Baer invariants, derived limits.

\section{Introduction and main results}
Serre's question on one-relator pro-\(p\) groups, as corrected by Gildenhuys \cite{serre,gildenhuys}, asks whether a group of cohomological dimension greater than two must admit a presentation \(F^{(p)}/(u^p)\). Here \(F^{(p)}\) is a finite-rank free pro-\(p\) group and \((u^p)\) is the closed normal subgroup generated by \(u^p\). The correction concerns the choice of normal generator: a prescribed relator need not itself be a \(p\)th power. We study the related problem of detecting higher homotopy in completed presentation complexes.

Throughout, \((r)\) means the normal subgroup generated by \(r\), with closure understood in the pro-\(p\) category. We use \(x^y=y^{-1}xy\) and \([x,y]=x^{-1}y^{-1}xy\). Let
\[
K=K(x_1,\ldots,x_d\mid r),\qquad 1\ne r\in F_d,
\]
where \(F_d\) is discrete and free, and \(K\) is the standard two-dimensional presentation complex. Write
\[
Z_K=\Z_\infty K,\qquad Q_K=\Q_\infty K,\qquad K_p^\wedge=(\F_p)_\infty K.
\]
These are completions of \emph{spaces} \cite{bk}. The last is modeled by degreewise pro-\(p\) completion of a free simplicial loop model. It is not identified in advance with the classifying space of the pro-\(p\) completion of \(\pi_1K\). The adelic space \(A_K\) is defined in Section~\ref{adelic}.

\begin{theorem}\label{main}
For every nontrivial relator \(r\), including a proper power,
\[
\pi_j(Q_K)=\pi_j(A_K)=0\qquad(j\ge2),
\]
and the natural maps induce isomorphisms of abelian groups
\begin{equation}\label{product}
\pi_j(Z_K)\cong\prod_p\pi_j(K_p^\wedge),\qquad j\ge2.
\end{equation}
\end{theorem}
The proof combines the arithmetic square of \cite[Proposition~4.2]{square} with the direct proof of rational asphericity in Section~\ref{ratasphsection}. Exact completed scalar extension from \(\Q\) to the finite adeles transfers that vanishing to the adelic entry.

If \(r\) is not a proper power, Lyndon's theorem \cite{lyndon} gives \(K\simeq BG\), where \(G=F_d/(r)\). In this case the group-theoretic description of \(\pi_2(Z_K)\) is
\begin{equation}\label{introseq}
0\longrightarrow\limone_nL_2\nu_n(G)\longrightarrow\pi_2(\Z_\infty BG)
\longrightarrow\limi_nM_n(G)\longrightarrow0,
\end{equation}
where \(\nu_n(G)=G/\gamma_{n+1}G\) and \(M_n(G)=L_1\nu_n(G)\). Section~\ref{baer} defines the functors and all transition maps. The middle term is the second homotopy group of the \emph{integral completion}; it is distinct from the zero group \(\pi_2(BG)\).

The next theorem applies to arbitrary finite connected two-complexes.
\begin{theorem}\label{mannanmain}
Let \(K\) be a finite connected two-complex. Let \((X,k)\) be either \((\Q_\infty K,\Q)\) or \((K_p^\wedge,\Z_p)\), and put \(H=\pi_1X\). There is a canonical \(H\)-equivariant isomorphism
\[
\pi_3X\cong\bigl(\pi_2X\widehat\otimes_k\pi_2X\bigr)^\tau,
\qquad\tau(x\otimes y)=y\otimes x,
\]
with diagonal action. The homotopy modules are linearly compact in the rational case and compact in the pro-\(p\) case, and the tensor products are taken in these categories. Moreover,
\[
\pi_2X=0\quad\Longrightarrow\quad\pi_jX=0\quad(j\ge2).
\]
No finiteness assumption on \(H\) is required, and the pro-\(p\) statement includes \(p=2\).
\end{theorem}
Appendix~\ref{rationalmannan} proves the rational statement through a direct cohomological Whitehead criterion and a free model on cycles. Appendix~\ref{quillenmannan} gives the pro-\(p\) argument using Quillen's criterion \cite[Corollary~2.9]{quillenprofinite}.

For the discrete one-relator groups to which Lyndon's asphericity theorem applies, the derived-limit term in \eqref{introseq} can be removed.
\begin{theorem}\label{baermain}
Let \(G=F_d/(r)\), where \(r\ne1\) is not a proper power. Then
\[
\limone_nL_2\nu_n(G)=0,
\]
and the natural comparison maps give isomorphisms
\begin{equation}\label{baeridentification}
\pi_2(\Z_\infty BG)\cong\limi_nM_n(G)
\cong\prod_p\pi_2((BG)_p^\wedge).
\end{equation}
\end{theorem}
The vanishing follows from rational asphericity, finite generation at each nilpotent stage, and a rational-to-integral Mittag--Leffler argument (Section~\ref{derivedvanishingsection}). Thus a compatible nonzero Baer class exists exactly when the integral completion has nonzero second homotopy.

The integral Fox comparison of Section~\ref{integralfox} strengthens Theorem~\ref{baermain}. Put \(\Lambda=\Z G\), let \(I\) be its augmentation ideal, and write \(\widehat\Lambda_I=\lim_n\Lambda/I^n\). The completed Fox operator is
\[
\widehat D_r(a)=\bigl(a\,\overline{\partial r/\partial x_i}\bigr)_{i=1}^d.
\]
\begin{theorem}\label{foxmain}
Under the hypotheses of Theorem~\ref{baermain}, the abelianizations of the relative relation quotients induce canonical topological isomorphisms
\[
\pi_2(\Z_\infty BG)\cong\limi_n M_n(G)
\cong\ker\widehat D_r
\cong\prod_p\ker D_{r,p},
\]
where \(D_{r,p}\) is the Fox operator on \(\Z_p[[G_p^\wedge]]\). These groups are profinite and torsion-free, and the comparisons respect the natural \(G\)-actions. The topology on the first group is the Moore quotient topology of the lower central completion model.
\end{theorem}
The comparison is induced by coefficient maps in a common cellular model. Section~\ref{integralfox} also describes the finite-stage kernels and their lifting obstructions.

Rational asphericity is known \cite[Corollary~12.7]{liemodels}; Section~\ref{ratasphsection} gives a proof using completed Fox derivatives. We use the arithmetic square of \cite{square}.

The remaining problem is to find a one-relator pro-\(p\) presentation with \(\pi_2(K_p^\wedge)\ne0\) whose closed relation subgroup admits no normal generator of the form \(u^p\).

\section{The adelic argument}\label{adelic}
\subsection{Models and the arithmetic square}
Choose the free cellular simplicial presentation model \(P_\bullet\) with \(WP_\bullet\simeq K\), finite free in each degree and with nondegenerate free generators only in degrees zero and one. We follow \cite[Sections~1.2 and~3]{square}: \(W\) denotes the reduced classifying-space construction, often written \(\overline W\). Its contractible total-space construction is a different functor. In particular,
\[
\pi_1(WH_\bullet)\cong\pi_0(H_\bullet),\qquad
\pi_j(WH_\bullet)\cong\pi_{j-1}(H_\bullet)\quad(j\ge2).
\]
Geometric realization is understood when a simplicial set is regarded as a space. Put
\[
N_{s,\bullet}=P_\bullet/\gamma_{s+1}P_\bullet.
\]
Let \(U_{s,m}\) be the unipotent \(\Q\)-group whose rational points are the rational nilpotent completion of \(N_{s,m}\). The completion models are
\[
\widehat P_{\Q,\bullet}=\limi_sU_{s,\bullet}(\Q),\qquad
\widehat P_{p,\bullet}=\limi_s\widehat N_{s,\bullet,p},
\]
with \(Q_K\simeq W\widehat P_{\Q,\bullet}\) and \(K_p^\wedge\simeq W\widehat P_{p,\bullet}\). Define
\begin{equation}\label{adeles}
A(P)_\bullet=\limi_s\prod_p'\bigl(U_{s,\bullet}(\Q_p),\widehat N_{s,\bullet,p}\bigr),
\qquad A_K=WA(P)_\bullet.
\end{equation}
At a fixed stage, the restricted product consists of families integral at almost every prime. The finite exceptional set may depend on the stage and the simplicial degree.

The arithmetic-square theorem \cite[Proposition~4.2]{square} gives a homotopy Cartesian square
\begin{equation}\label{squareeq}
\begin{tikzcd}
Z_K\arrow[r]\arrow[d]&Q_K\arrow[d]\\
\displaystyle\prod_pK_p^\wedge\arrow[r]&A_K.
\end{tikzcd}
\end{equation}
The order of limits in \eqref{adeles} is part of the construction: the restricted product precedes the inverse limit. The cited theorem is applied to the same simplicial model and its natural comparison maps. For \(j\ge2\), its exact sequence contains
\begin{equation}\label{les}
\pi_{j+1}(A_K)\longrightarrow\pi_j(Z_K)\longrightarrow
\pi_j(Q_K)\oplus\prod_p\pi_j(K_p^\wedge)\longrightarrow\pi_j(A_K)
\end{equation}

\subsection{The cohomological-dimension criterion}\label{cdsection}
\begin{proposition}\label{cdcriterion}
Let \(K=K(x_1,\ldots,x_d\mid r)\), with \(1\ne r\in F_d\), and let \(G=F_d/(r)\). For every prime \(p\),
\[
\pi_1(K_p^\wedge)\cong G_p^\wedge,\qquad
\pi_2(K_p^\wedge)=0\quad\Longleftrightarrow\quad\cdp(G_p^\wedge)\le2.
\]
The statement includes proper-power relators and does not require the presentation to be minimal on generators.
\end{proposition}
\begin{proof}
Put \(F=F_d^{(p)}\), \(R=(r)\le F\), and \(H=F/R=G_p^\wedge\). The finite-type completion model identifies \(K_p^\wedge\) with \(W\widehat P_{p,\bullet}\), where \(\widehat P_{p,\bullet}\) is the free pro-\(p\) simplicial presentation of \(H\). Its zeroth homotopy group is \(H\), and its first is the second homotopy group of that presentation.

The criterion of \cite[Section~5.3, Propositions~15 and~16, with the pro-\(p\) convention of Section~4.5]{mikhovich2025} applies once the chosen relator is a basis of \(R/\overline{R^p[R,F]}\). Here this condition is automatic: normal generation makes that quotient at most one-dimensional over \(\F_p\), and it is nonzero. Indeed, \(r\ne1\) in the residually pro-\(p\) free group, so choose a finite \(p\)-group quotient \(F\to E\) in which its image survives. The nontrivial image \(S\) of \(R\) has
\(S/(S^p[S,E])\ne0\), by the nilpotence of the augmentation ideal of \(\F_p[E]\) acting on \(S/\Phi(S)\). Thus the class of \(r\) is the required basis.

Set \(\Lambda=\Z_p[[H]]\), \(J=\ker(\Lambda\to\F_p)\), and \(R_{\rm ab}=R/\overline{[R,R]}\). The completed relation and identity sequences are
\[
0\longrightarrow R_{\rm ab}\longrightarrow\Lambda^d
\longrightarrow\Lambda\longrightarrow\Z_p\longrightarrow0,
\]
\[
0\longrightarrow\pi_2(K_p^\wedge)\longrightarrow\Lambda
\xrightarrow{\,1\mapsto\bar r\,}R_{\rm ab}\longrightarrow0.
\]
If \(\cdp H\le2\), dimension shifting makes \(R_{\rm ab}\) projective as a compact \(\Lambda\)-module. The second sequence splits. Its map on quotients modulo \(J\) is an isomorphism by the preceding basis calculation; hence the kernel has zero quotient modulo \(J\). Compact Nakayama gives \(\pi_2(K_p^\wedge)=0\). Conversely, if that kernel vanishes, the first sequence is a free resolution of \(\Z_p\) of length at most two, so \(\cdp H\le2\).
\end{proof}

\subsection{Completed scalar extension}
A linearly compact rational vector space is an inverse limit \(V=\limi_sV_s\) of finite-dimensional discrete spaces with surjective transition maps. For a commutative \(\Q\)-algebra \(E\), set
\[
T_E(V)=V\widehat\otimes_\Q E:=\limi_s(V_s\otimes_\Q E),
\]
using the discrete topology on the factors. The definition is independent of the chosen cofinal system of open finite-codimension subspaces. When \(E\) is infinite-dimensional, \(T_E(V)\) need not be linearly compact over \(\Q\); exactness below concerns its underlying vector space.

\begin{lemma}\label{exact}
The functor \(T_E\) is exact on linearly compact rational vector spaces with continuous linear maps. For a complex \(C_\bullet\) of such spaces,
\begin{equation}\label{homexact}
H_q(T_EC_\bullet)\cong T_E(H_qC_\bullet).
\end{equation}
\end{lemma}
\begin{proof}
Continuous duality gives \(V^\vee=\varinjlim_sV_s^*\) and
\[
T_E(V)=\Hom_\Q(V^\vee,E).
\]
Continuous duality is exact, and \(\Hom_\Q(-,E)\) is exact because short exact sequences of vector spaces split. Images of continuous maps between linearly compact spaces are closed. Hence cycles, boundaries, and homology remain in the source category, and exactness applied to their defining sequences gives \eqref{homexact}.
\end{proof}

\subsection{Logarithms and Moore complexes}
Write \(\A=\prod_p'(\Q_p,\Z_p)\), \(L_{s,\bullet}=\Lie(U_{s,\bullet})\), and \(L_\bullet=\limi_sL_{s,\bullet}=\log\widehat P_{\Q,\bullet}\). Each \(L_{s,m}\) is finite-dimensional. At fixed \(s,m\), Hall coordinates and the Baker--Campbell--Hausdorff formulas have denominators involving finitely many primes. For a full lattice \(\Lambda_{s,m}\subset L_{s,m}\), outside a finite set of places one has
\[
\log\widehat N_{s,m,p}=\Lambda_{s,m}\otimes_\Z\Z_p.
\]
Changing integrality at finitely many places does not change a restricted product. Consequently,
\[
\log\prod_p'\bigl(U_{s,m}(\Q_p),\widehat N_{s,m,p}\bigr)
\cong L_{s,m}\otimes_\Q\A.
\]
The identification is natural in the rational homomorphisms induced by faces, degeneracies, and tower maps. Therefore
\begin{equation}\label{logadel}
\log A(P)_\bullet\cong T_{\A}(L_\bullet).
\end{equation}
All completeness statements here refer to the nilpotent-stage filtration; a common exceptional set of primes is unnecessary.

For a simplicial Lie algebra, put
\[
N_qL=\bigcap_{i=1}^q\ker d_i,\qquad\partial=d_0.
\]
Exactness preserves finite intersections of kernels, so Lemma~\ref{exact} gives
\begin{equation}\label{normalization}
N_\bullet(T_EL)\cong T_E(N_\bullet L),\qquad
H_qN_\bullet(T_EL)\cong T_E(H_qN_\bullet L).
\end{equation}
For the simplicial group \(H_\bullet=\exp L_\bullet\), Moore cycles and boundaries are
\[
Z_qH=\bigcap_{i=0}^q\ker d_i,\qquad B_qH=d_0(N_{q+1}H).
\]
Exponentiation commutes with morphisms and gives bijections
\(Z_qH=\exp(Z_qNL)\) and \(B_qH=\exp(B_qNL)\). Thus, for \(q\ge1\),
\[
\pi_q(\exp L_\bullet)=0\quad\Longleftrightarrow\quad H_qN_\bullet L=0.
\]
The same identities hold for the complete BCH algebra \(T_{\A}L_\bullet\). Together with \eqref{logadel}--\eqref{normalization}, they prove:

\begin{proposition}\label{transfer}
For every \(j\ge2\), \(\pi_j(Q_K)=0\) implies \(\pi_j(A_K)=0\).
\end{proposition}

\subsection{Rational asphericity}\label{ratasphsection}
We use the completed Hopf algebra \(\mathcal O(H)^*\) and its prounipotent interpretation as in \cite[Sections~1.2 and~2.1]{identity}. We first record a characteristic-zero property of this algebra. It will supply the vanishing of \(\pi_2\) needed by the rational argument of Appendix~\ref{rationalmannan}.

\begin{lemma}\label{ratdomain}
For a prounipotent group \(H\) over \(\Q\), the convolution algebra \(\mathcal O(H)^*\) has no zero divisors.
\end{lemma}
\begin{proof}
First let \(H\) be finite-dimensional, with nilpotent Lie algebra \(\mathfrak h\). Choose a basis adapted to its lower central series and assign each basis vector its lower central weight. The PBW monomials give the augmentation-adic filtration of \(U(\mathfrak h)\): its \(n\)th term is spanned by monomials of weight at least \(n\). Indeed, a vector of weight \(w\) is a sum of iterated brackets of length \(w\), and commutation into PBW order never lowers total weight. Consequently
\[
\operatorname{gr}_I U(\mathfrak h)\cong
U(\operatorname{gr}_{\gamma}\mathfrak h).
\]
The algebra on the right has no zero divisors: its ordinary PBW filtration has associated graded algebra \(\operatorname{Sym}(\operatorname{gr}_{\gamma}\mathfrak h)\), a polynomial domain. The weighted filtration is separated and has finite-dimensional quotients. Thus its completion is a domain as well, by multiplying the first nonzero homogeneous terms.

In exponential coordinates, polynomial functions are exhausted by finite weighted degree. Duality between functions and PBW distributions identifies \(\mathcal O(H)^*\) with this completed enveloping algebra; see also the complete Hopf-algebra correspondence in \cite[Appendix~A]{quillenrational}.

In general write \(H=\limi_\alpha H_\alpha\) with finite-dimensional unipotent quotients. Then \(\mathcal O(H)^*=\limi_\alpha\mathcal O(H_\alpha)^*\). If two elements are nonzero, their images are both nonzero in some common quotient algebra, which is a domain by the preceding argument. Their product is therefore nonzero.
\end{proof}

\begin{proposition}\label{rataspherical}
If \(K=K(x_1,\ldots,x_d\mid r)\) with \(1\ne r\in F_d\), then
\[
\pi_j(\Q_\infty K)=0\qquad(j\ge2).
\]
Proper powers are allowed.
\end{proposition}
\begin{proof}
Let \(F\) be the free prounipotent group on the \(x_i\), \(H=F/(r)\), and \(\Lambda=\mathcal O(H)^*\). The completed algebra of \(F\) is
\[
A=\Q\langle\!\langle T_1,\ldots,T_d\rangle\!\rangle,
\qquad T_i=x_i-1,
\]
and \(\Lambda=A/\overline{A(r-1)A}\). The latter follows equally by taking the completed enveloping algebra of the quotient Lie algebra: \(\log r\) and \(r-1\) generate the same closed two-sided ideal.

The Magnus expansion is injective on the discrete free group, so \(r-1\ne0\). Let \(w\ge1\) be its lowest degree. Define the continuous Fox derivatives by the unique last-letter decomposition
\[
r-1=\sum_{i=1}^d D_i(r)T_i.
\]
At least one \(D_i(r)\) has a nonzero term of degree \(w-1\). The entire closed ideal \(\overline{A(r-1)A}\) lies in degree at least \(w\); hence this derivative has nonzero image \(\overline{D_i(r)}\) in \(\Lambda\).

Lemma~\ref{ratkernel}, applied to the rational simplicial presentation, identifies
\[
\pi_2(\Q_\infty K)=\ker\left(
\Lambda\longrightarrow\Lambda^d,\quad
\lambda\longmapsto
(\lambda\overline{D_1(r)},\ldots,\lambda\overline{D_d(r)})
\right).
\]
Since \(\Lambda\) is a domain and one coordinate is nonzero, this map is injective. The rational free model of Appendix~\ref{rationalmannan} now gives vanishing of every higher homotopy group.
\end{proof}

\begin{proof}[Proof of Theorem~\ref{main}]
Proposition~\ref{rataspherical} and adelic transfer give \(\pi_j(Q_K)=\pi_j(A_K)=0\) for \(j\ge2\). The two outer groups and the rational summand in \eqref{les} vanish, proving \eqref{product}.
\end{proof}

\begin{lemma}\label{prophomotopy}
For \(j\ge2\), \(\pi_j(K_p^\wedge)\) carries a natural abelian pro-\(p\) group structure. In particular, multiplication by a prime \(\ell\ne p\) is an automorphism.
\end{lemma}
\begin{proof}
The simplicial group \(\widehat P_{p,\bullet}\) consists of compact pro-\(p\) groups and continuous maps. Its Moore cycles are closed, and its Moore boundaries are compact images, hence closed. Their quotient is therefore pro-\(p\). It is abelian in positive simplicial degrees. Multiplication by \(\ell\) is invertible on every finite abelian \(p\)-group and hence on their inverse limit.
\end{proof}

\section{Baer invariants and the second homotopy group}\label{baer}
\subsection{Derived functors and transition maps}
Let \(G\) be a discrete group. A free simplicial resolution \(F_\bullet\to G\) is a free simplicial group, with bases compatible with degeneracies, such that
\[
\pi_0F_\bullet\cong G,\qquad\pi_iF_\bullet=0\quad(i>0).
\]
The augmentation is to the constant simplicial group \(G\); one may use the free-group cotriple resolution. For \(\nu_n(G)=G/\gamma_{n+1}G\), define the nonadditive derived functors by
\begin{equation}\label{deriveddef}
L_i\nu_n(G)=\pi_i\bigl(\nu_n(F_\bullet)\bigr).
\end{equation}
They are independent of the free resolution up to natural isomorphism. The index \(i\) is the derived degree, while \(n\) is the nilpotent class. We have \(L_0\nu_n(G)=\nu_n(G)\); for \(i\ge1\), the groups are abelian. Explicitly, for a simplicial group \(H_\bullet\),
\[
N_iH=\bigcap_{k=1}^i\ker d_k,\qquad
\pi_iH_\bullet=\ker(d_0|_{N_iH})/d_0(N_{i+1}H)\quad(i\ge1).
\]
The natural epimorphisms \(\nu_{n+1}\to\nu_n\) induce the transition maps in each tower \(L_i\nu_n(G)\).

The \(n\)th nilpotent Baer invariant, also called the \(n\)-nilpotent multiplier, is \(M_n(G)=L_1\nu_n(G)\). For a free presentation \(G=F/R\), the generalized Hopf formula is
\begin{equation}\label{hopf}
M_n(G)=\frac{R\cap\gamma_{n+1}F}{[R,{}_nF]},
\qquad [R,{}_0F]=R,\quad[R,{}_{n+1}F]=[[R,{}_nF],F].
\end{equation}
See \cite{burnsellis} and \cite[Section~2]{right}. In particular, \(M_1(G)=H_2(G;\Z)\), and \(L_i\nu_1(G)=H_{i+1}(G;\Z)\) for \(i\ge1\). In \eqref{hopf}, the map \(M_{n+1}(G)\to M_n(G)\) is induced by the inclusions of the numerator and denominator. It need be neither injective nor surjective.

\subsection{Limits and the Milnor sequence}
For a countable tower of abelian groups \((A_n,f_n)\), let
\[
\Delta:\prod_nA_n\to\prod_nA_n,\qquad
\Delta((a_n))=(a_n-f_n(a_{n+1}))_n.
\]
Then \(\limi A_n=\ker\Delta\) and \(\limone A_n=\coker\Delta\). Thus the ordinary limit consists of compatible families. If the images \(\im(A_m\to A_n)\) stabilize for each \(n\), the tower satisfies the Mittag--Leffler condition and \(\limone A_n=0\). This applies to towers of finite groups and to towers with surjective transitions.

\begin{proposition}\label{milnor}
Every discrete group \(G\) has a natural short exact sequence \eqref{introseq}.
\end{proposition}
\begin{proof}
The lower central series model is \(\Z_\infty BG\simeq W\holim_n\nu_n(F_\bullet)\). The tower maps are degreewise surjective, hence fibrations of simplicial groups. The Milnor sequence for its first homotopy group reads
\[
0\to\limone_n\pi_2\nu_n(F_\bullet)\to\pi_1\holim_n\nu_n(F_\bullet)
\to\limi_n\pi_1\nu_n(F_\bullet)\to0.
\]
Using \eqref{deriveddef} and the degree shift under \(W\) proves the result.
\end{proof}
\subsection{Removing the derived-limit term}\label{derivedvanishingsection}
\begin{lemma}\label{rationalML}
Let \(A_1\leftarrow A_2\leftarrow\cdots\) be a tower of finitely generated abelian groups. If
\[
\limi_n(A_n\otimes\Q)=0,
\]
then \((A_n)\) is Mittag--Leffler and \(\limone_nA_n=0\). More precisely, for each \(n\) some transition \(A_m\to A_n\), \(m\ge n\), has image in the finite subgroup \(\operatorname{Tor}(A_n)\).
\end{lemma}
\begin{proof}
Set \(V_n=A_n\otimes\Q\). In each finite-dimensional \(V_n\), the descending images of the \(V_m\) stabilize. The stable images form a tower with surjective transitions and have the same inverse limit as \((V_n)\). A compatible lift of any element in a stable image can be chosen successively, so that stable image is exactly the image of the inverse limit in \(V_n\). It is therefore zero. Hence some \(V_m\to V_n\) is zero, and \(\im(A_m\to A_n)\subseteq\operatorname{Tor}(A_n)\). All subsequent images form a descending chain in a finite group, so they stabilize. The countable Mittag--Leffler criterion gives \(\limone A_n=0\). Every compatible family then has all its coordinates in \(\operatorname{Tor}(A_n)\), so \(\limi_n A_n=\limi_n\operatorname{Tor}(A_n)\).
\end{proof}

\begin{proposition}\label{derivedvanishing}
Suppose \(G\) admits a degreewise finite free simplicial resolution and \(\pi_3(\Q_\infty BG)=0\). Then \((L_2\nu_n(G))_n\) is Mittag--Leffler,
\[
\limone_nL_2\nu_n(G)=0,
\qquad
\pi_2(\Z_\infty BG)\xrightarrow{\ \cong\ }\limi_nM_n(G).
\]
The second isomorphism is the natural Milnor edge map.
\end{proposition}
\begin{proof}
Choose such a resolution \(F_\bullet\to G\). Each \(\nu_n(F_j)\) is finitely generated nilpotent. Its Moore-cycle subgroups are finitely generated, as are their quotients. Thus every \(L_i\nu_n(G)\), \(i\ge1\), is finitely generated abelian.

The finite-type rational completion model is
\[
\Q_\infty BG\simeq W\limi_n X_{n,\bullet},
\qquad X_{n,\bullet}=\nu_n(F_\bullet)\otimes\Q,
\]
with rationalization taken degreewise in the nilpotent category. Exactness of nilpotent rationalization \cite{hmr} identifies
\[
\pi_iX_{n,\bullet}\cong L_i\nu_n(G)\otimes\Q
\quad(i\ge1).
\]
The transition maps are degreewise surjective. The degree-two Milnor sequence for this tower therefore gives
\[
0\longrightarrow\limone_n\pi_3X_{n,\bullet}
\longrightarrow\pi_3(\Q_\infty BG)
\longrightarrow\limi_n(L_2\nu_n(G)\otimes\Q)
\longrightarrow0.
\]
The middle term is zero, hence so is the last. Lemma~\ref{rationalML} proves the first assertion, and \eqref{introseq} gives the asserted edge isomorphism. All maps are natural; a functorial free resolution may be used to express that naturality.
\end{proof}

\begin{proof}[Proof of Theorem~\ref{baermain}]
Lyndon's theorem gives \(K\simeq BG\) when \(r\) is not a proper power. Its finite cellular loop model is a degreewise finite free resolution. Proposition~\ref{rataspherical} gives \(\pi_3(\Q_\infty BG)=0\), so Proposition~\ref{derivedvanishing} applies. Combine its edge isomorphism with Theorem~\ref{main}.
\end{proof}
The same reasoning, using \(\pi_{i+1}(\Q_\infty BG)=0\), shows in this one-relator case that \(\limone_nL_i\nu_n(G)=0\) for every \(i\ge1\). In particular the fundamental-group correction \(\limone_nM_n(G)\) also vanishes. Thus the natural map gives
\[
\pi_1(\Z_\infty BG)\cong\widehat G_{\rm nil}.
\]
 None of these statements asserts that the individual derived functors vanish, or that their integral towers are pro-zero. An integral inverse limit can remain nonzero through compatible torsion classes.

\subsection{Local detection and a kernel description}
Under the hypotheses of Theorem~\ref{baermain}, the local form of the Baer-limit identification is
\begin{equation}\label{localmilnor}
\prod_p\pi_2((BG)_p^\wedge)\xrightarrow{\ \cong\ }\limi_nM_n(G).
\end{equation}
In particular, a finite \(p\)-group image singles out a prime.
\begin{corollary}\label{fixedprime}
If there is a nonzero homomorphism \(\limi_nM_n(G)\to C\), where \(C\) is a finite abelian \(p\)-group, then \(\pi_2((BG)_p^\wedge)\ne0\).
\end{corollary}
\begin{proof}
Compose \eqref{localmilnor} with the given homomorphism. By Lemma~\ref{prophomotopy}, multiplication by \(p\) is invertible on \(\prod_{q\ne p}\pi_2((BG)_q^\wedge)\). Every homomorphism from this product to a bounded-exponent \(p\)-group is zero. The nonzero composite must therefore be nonzero on the \(p\)-factor.
\end{proof}

For an arbitrary group, the different limit \(\limone_nM_n(G)\) measures the fundamental-group correction \cite[Proposition~3.3]{right}:
\[
1\to\limone_nM_n(G)\to\pi_1(\Z_\infty BG)\to\widehat G_{\rm nil}\to1,
\qquad\widehat G_{\rm nil}=\limi_n\nu_n(G).
\]
For the one-relator groups of Theorem~\ref{baermain}, the observation following Proposition~\ref{derivedvanishing} makes this correction zero as well.

There is also a direct description of the ordinary Baer limit. For a free presentation \(G=F/R\), put \(R_n=[R,{}_nF]\) and
\[
\widehat R_F=\limi_{n\ge1}R/R_n.
\]
This relative completion is used in \cite[Section~3, Proposition~3.13]{right}. Inclusion in \(F\) induces a map to \(\widehat F_{\rm nil}\), and the Hopf formula gives
\begin{equation}\label{relativekernel}
\ker(\widehat R_F\to\widehat F_{\rm nil})
=\limi_n\ker(R/R_n\to F/\gamma_{n+1}F)
=\limi_nM_n(G).
\end{equation}
Only preservation of kernels by inverse limits is used here.

\section{Integral Fox kernels and local comparison}\label{integralfox}
Throughout this section \(G=F_d/(r)\), where \(r\ne1\) is not a proper power, and \(K\simeq BG\) is its presentation complex. Set
\[
\Lambda=\Z G,\quad I=\ker(\Lambda\xrightarrow{\varepsilon}\Z),\quad
A_n=\Lambda/I^n,\quad \widehat\Lambda_I=\limi_n A_n.
\]
Write \(d_i\in\Lambda\) for the image of the Fox derivative \(\partial r/\partial x_i\), with the convention \(r-1=\sum_i(\partial r/\partial x_i)(x_i-1)\). Define
\[
D_{r,n}:A_n\longrightarrow A_n^d,\quad a\longmapsto(a d_i)_i,
\qquad E_n=\ker D_{r,n},\qquad E=\limi_n E_n.
\]
All the maps are left \(\Lambda\)-linear; coefficient modules in group homology below are right modules. Preservation of kernels by inverse limits gives
\begin{equation}\label{completedfoxkernel}
E=\ker\bigl(\widehat D_r:\widehat\Lambda_I\longrightarrow\widehat\Lambda_I^{\,d}\bigr).
\end{equation}
The factors \(A_n\) and \(E_n\) carry the discrete topology. Neither \(\widehat\Lambda_I\) nor the lower central completion of a group is being identified with a product of pro-\(p\) completions.

\subsection{The comparison from the Baer tower}
Use the cellular model \(P_\bullet\) of Section~\ref{adelic}. In degrees zero and one it is \(F\) and \(U=F*\langle y\rangle\). For this description choose faces
\[
d_0(y)=1,\quad d_1(y)=r,\qquad d_0(x_i)=d_1(x_i)=x_i,
\]
and put \(H=\ker d_0\), \(J=\ker d_1\). The two choices of end face for the Moore differential give the same homotopy groups; here we use \(N_q=\bigcap_{i<q}\ker d_i\), with differential \(d_q\). The free crossed module is \(H/[H,J]\to F\). Lyndon's asphericity theorem identifies it with the inclusion \(R\hookrightarrow F\), where \(R=(r)\).

\begin{lemma}\label{relativefox}
Put \(U_n=U/\gamma_{n+1}U\), \(F_n=F/\gamma_{n+1}F\), \(H_n=\ker d_{0,n}\), \(J_n=\ker d_{1,n}\), and \(C_n=H_n/[H_n,J_n]\). Then, naturally in \(n\),
\[
Q_n:=d_1(H\cap\gamma_{n+1}U)=[R,{}_nF],\qquad
C_n\cong R/[R,{}_nF],\qquad C_n^{\rm ab}\cong A_n.
\]
The kernel of \(C_n\to F_n\) is \(B_n=M_n(G)\). Abelianization induces maps
\begin{equation}\label{betamaps}
\beta_n:B_n\longrightarrow E_n
\end{equation}
compatible with the transition maps. Consequently the Milnor edge map defines a continuous homomorphism
\[
\rho:\pi_2(\Z_\infty K)\longrightarrow E.
\]
\end{lemma}
\begin{proof}
For the splitting \(U=H\rtimes F\), commutator identities give
\[
\gamma_{n+1}U=[H,{}_nU]\rtimes\gamma_{n+1}F.
\]
One proves this by induction: the mixed commutators are absorbed by \([H,{}_nU]\), using \([H,\gamma_jU]\subseteq[H,{}_jU]\). Intersecting with \(H\) and applying \(d_1\), for which \(d_1H=R\) and \(d_1U=F\), proves the assertion about \(Q_n\).

Both faces split over \(F\). Thus \(H\to H_n\) and \(J\to J_n\) are onto, and
\[
C_n\cong H/\bigl([H,J](H\cap\gamma_{n+1}U)\bigr)\cong R/Q_n.
\]
The free cellular simplicial continuation has Moore boundaries \([H_n,J_n]\) in degree one, so its first homotopy group is \(\ker(C_n\to F_n)\). Equivalently, the Hopf formula gives this kernel as \(M_n(G)\).

Under the Lyndon identification \(R_{\rm ab}\cong\Lambda\), with \(\bar r\mapsto1\), the image of \([R,{}_nF]\) is \(I^nR_{\rm ab}\). This proves \(C_n^{\rm ab}\cong A_n\). We use the left action for which conjugation \(r\mapsto f r f^{-1}\) corresponds to multiplication by the image of \(f\) in \(G\).

If \(b\in B_n\) is represented by \(w\in R\cap\gamma_{n+1}F\), its coefficient \(a\in R_{\rm ab}\cong\Lambda\) satisfies
\[
\overline{\partial w/\partial x_i}=a d_i.
\]
Fox derivatives of \(\gamma_{n+1}F\) lie in the \(n\)th power of the augmentation ideal of \(\Z F\). Hence \(a d_i\in I^n\), proving \eqref{betamaps}. All constructions commute with quotient maps. Continuity of \(\rho\) follows from its coordinate maps on Moore cycles.
\end{proof}

\begin{remark}\label{finitekernelremark}
Let \(R^{(n)}=R\gamma_{n+1}F/\gamma_{n+1}F\). The extension
\(1\to B_n\to C_n\to R^{(n)}\to1\) is central, and its five-term homology sequence gives
\[
\ker\beta_n=B_n\cap[C_n,C_n]
=\im\bigl(H_2(R^{(n)};\Z)\longrightarrow B_n\bigr).
\]
The comparison need not be an isomorphism at a finite stage. Our theorem concerns the inverse limit of these particular maps.
\end{remark}

\subsection{Finite coefficients and compactness}
Each \(A_n\) is a finitely generated abelian group: monomials of length less than \(n\) in \(x_i-1\) span it, since inverse generators have finite geometric expansions modulo \(I^n\). The same is true of \(E_n\).

\begin{lemma}\label{groupalgebracomparison}
For any finitely generated group \(G\) and every prime \(p\), there are natural topological ring isomorphisms
\begin{equation}\label{ringlocalcomparison}
\limi_n(A_n\otimes_\Z\Z_p)
\cong\limi_{n,k}\Lambda/(I^n+p^k\Lambda)
\cong\Z_p[[G_p]],
\end{equation}
where \(G_p=G_p^\wedge\), and the tensor factors on the left have their \(p\)-adic topologies.
\end{lemma}
\begin{proof}
Finite generation of \(A_n\) identifies \(A_n\otimes\Z_p\) with \(\lim_k A_n/p^kA_n\). In the finite ring \(\Lambda/(I^n+p^k\Lambda)\), the image of \(I\) is nilpotent and the group \(1+I\) is a finite \(p\)-group. Thus the map from \(G\) factors through a finite \(p\)-group quotient. Conversely, for a finite \(p\)-group \(P\), the augmentation ideal of \((\Z/p^k)[P]\) is nilpotent: it is nilpotent modulo \(p\), and \(p^k=0\). The two systems of finite ring quotients are therefore cofinal. This proves the second isomorphism as well as its topological and natural character.
\end{proof}

\begin{proposition}\label{foxML}
The tower \((E_n)\) is Mittag--Leffler and
\[
\limone_n E_n=0,\qquad E=\limi_n\operatorname{Tor}(E_n).
\]
For \(T_p:=\limi_n\operatorname{Tor}_p(E_n)\), the natural maps give topological isomorphisms
\begin{equation}\label{foxlocaliso}
E\cong\prod_pT_p,\qquad
T_p\cong\ker\bigl(D_{r,p}:\Z_p[[G_p]]\longrightarrow\Z_p[[G_p]]^d\bigr).
\end{equation}
In particular, \(E\) is profinite and torsion-free. No restriction on the exponent-sum vector of \(r\) is required.
\end{proposition}
\begin{proof}
The algebra \(\lim_n(A_n\otimes\Q)\) is the complete algebra of the rational prounipotent completion of \(G\). This follows either from the complete Hopf-algebra correspondence or from the quotient description used in Proposition~\ref{rataspherical}: completion of \(\Q G\) is the noncommutative power-series algebra modulo the closed ideal generated by \(r-1\). The proof of that proposition shows that the Fox operator on this algebra is injective. Flatness of \(\Q\) and preservation of kernels by inverse limits yield
\[
\limi_n(E_n\otimes\Q)=0.
\]
Lemma~\ref{rationalML} applies. For every fixed \(n\), a sufficiently late map \(E_m\to E_n\) has finite image in \(\operatorname{Tor}(E_n)\); hence every element of \(E\) has torsion coordinates.

Flatness of \(\Z_p\) and Lemma~\ref{groupalgebracomparison} identify \(\lim_n(E_n\otimes\Z_p)\) with \(\ker D_{r,p}\). The finite-image property just proved implies that every coordinate of a compatible family in this limit lies in \(\operatorname{Tor}_p(E_n)\). Thus the natural map \(T_p\to\ker D_{r,p}\) is a continuous bijection from a compact space to a Hausdorff space, and is a homeomorphism. Primary decomposition of each finite group \(\operatorname{Tor}(E_n)\) proves the product formula.

Finally, \(\Z_p[[G_p]]=\lim_V\Z_p[G_p/V]\) has no \(p\)-torsion, since each term is free over \(\Z_p\). It follows that \(T_p\) is torsion-free over \(\Z_p\), and that \(\prod_pT_p\) has no nonzero element of finite order.
\end{proof}

\begin{remark}
Even without the rational injectivity argument, finite generation and flatness give a natural injection
\[
\limi_n\operatorname{Tor}_p(E_n)\hookrightarrow\limi_n(A_n\otimes\Z_p)
\cong\Z_p[[G_p]].
\]
Thus absence of \(p\)-torsion in \(T_p\) does not depend on the arithmetic square. The square will identify these algebraic groups with the local homotopy groups.
\end{remark}

\subsection{The natural topological comparison}
We specify the topology on \(\Pi=\pi_2(\Z_\infty K)\) by its Moore-cycle quotient in the degreewise lower central completion \(\widehat P_\bullet=\lim_n\nu_nP_\bullet\), with discrete groups at each nilpotent stage.

\begin{lemma}\label{milnortopology}
For a countable tower of simplicial groups with degreewise surjective transitions and discrete stages, a bijective Milnor edge map
\[
\pi_1(\limi_n X_{n,\bullet})\longrightarrow\limi_n\pi_1(X_{n,\bullet})
\]
is a homeomorphism for the Moore quotient topology on the source. In our setting,
\begin{equation}\label{compactmilnor}
\Pi\cong_{\rm top}\limi_n\operatorname{Tor}(M_n(G)).
\end{equation}
The natural map \(\Pi\to\prod_p\pi_2(K_p^\wedge)\) is a homeomorphism as well.
\end{lemma}
\begin{proof}
A degreewise surjective morphism of simplicial groups is surjective on normalized groups: a lift can be corrected by degeneracies, or by filling the corresponding horn in the simplicial kernel. Successive lifts show that the projection from the inverse limit is also surjective on normalized groups.

Write \(Z\) for the degree-one Moore cycles of the inverse limit, \(D\) for its Moore boundaries, and \(q:Z\to Z/D\) for the quotient. Let \(W_n=\ker(Z\to Z_1X_{n,\bullet})\). If a cycle is a boundary at stage \(n\), lift a normalized degree-two filling to the inverse limit and remove its boundary. This gives
\[
q(W_n)=\ker\bigl(\pi_1(\lim X_{n,\bullet})\longrightarrow\pi_1(X_{n,\bullet})\bigr).
\]
The groups \(W_n\) form a neighbourhood basis in \(Z\), and the quotient map \(q\) is open. Under the assumed bijection, these neighbourhoods give exactly the inverse-limit topology. In particular \(D\) is closed in \(Z\).

For \(X_{n,\bullet}=\nu_nP_\bullet\), Theorem~\ref{baermain} supplies the bijective edge map. Rational asphericity and the rational Milnor sequence also give \(\lim_n(M_n(G)\otimes\Q)=0\). Lemma~\ref{rationalML} therefore reduces the limit to its finite torsion subgroups, proving \eqref{compactmilnor}.

Every finite \(p\)-group quotient of \(P_j\) factors through some nilpotent quotient. Thus the degreewise comparison
\(\widehat P_\bullet\to\prod_p\widehat P_{p,\bullet}\) is continuous. Its homotopy map is bijective by Theorem~\ref{main}. The source is compact by \eqref{compactmilnor}; the target is Hausdorff, since Moore boundaries in simplicial pro-\(p\) groups are compact images. The map is consequently a homeomorphism.
\end{proof}

\begin{theorem}\label{foxtheorem}
The map \(\rho\) of Lemma~\ref{relativefox} is a canonical topological isomorphism
\begin{equation}\label{fullfoxcomparison}
\pi_2(\Z_\infty K)\xrightarrow[\rho]{\ \cong\ }
\ker\bigl(\widehat D_r:\widehat\Lambda_I\longrightarrow\widehat\Lambda_I^{\,d}\bigr).
\end{equation}
It is compatible with the action of \(G\), the Milnor edge map and the local Fox identifications. In particular the following diagram commutes, and every arrow is a topological isomorphism:
\[
\begin{tikzcd}[column sep=large]
\Pi\arrow[r,"\rho"]\arrow[d]&E\arrow[d]\\
\displaystyle\prod_p\pi_2(K_p^\wedge)\arrow[r]&\displaystyle\prod_p\ker D_{r,p}.
\end{tikzcd}
\]
\end{theorem}
\begin{proof}
The local coefficient complex of Lemma~\ref{kernelhomology} identifies \(\pi_2(K_p^\wedge)\) with \(\ker D_{r,p}\); the same identification follows from the two relation and identity sequences in Proposition~\ref{cdcriterion}. By Lemma~\ref{milnortopology} and Proposition~\ref{foxML}, three sides of the displayed diagram are already topological isomorphisms. It remains to verify the diagram for the specified map \(\rho\).

Let \(\theta:H\to\Lambda\) be the coefficient homomorphism obtained from \(H\to R\to R_{\rm ab}\cong\Lambda\). Explicitly, if
\[
h=\prod_{j=1}^{t}f_jy^{\epsilon_j}f_j^{-1},\qquad f_j\in F,\quad\epsilon_j=\pm1,
\]
then \(\theta(h)=\sum_j\epsilon_j\bar f_j\), with \(\bar f_j\) the image in \(G\). This is independent of the expression for \(h\). The integral comparison reduces this coefficient modulo \(I^n\); Lemma~\ref{relativefox} shows that it factors through \(H_n\). These maps extend to \(\widehat H=\ker(d_0:\widehat U\to\widehat F)\) with values in \(\widehat\Lambda_I\).

The local coefficient map is obtained by abelianizing the free pro-\(p\) crossed module. Its value on the same \(h\) is the image of \(\sum_j\epsilon_j\bar f_j\) in \(\Z_p[[G_p]]\). There is a natural continuous ring map \(\widehat\Lambda_I\to\Z_p[[G_p]]\): at every finite coefficient quotient \(\Lambda/(I^n+p^k\Lambda)\) it is just reduction. Hence the integral and local coefficient maps agree on \(H\) after reduction modulo every \((I^n,p^k)\).

The projections \(H\to H_n\) are onto, so \(H\) is dense in \(\widehat H=\lim_n H_n\). Both maps to each finite coefficient quotient are continuous, and therefore agree on \(\widehat H\). Passing to all finite quotients proves their equality in \(\Z_p[[G_p]]\). Restricting to Moore cycles and then taking homotopy classes yields exactly the asserted square: on the integral side the coordinate is \(\beta_n\), and on the local side it is the Fox coefficient of the local cycle. Both kill Moore boundaries, since they factor through the respective crossed-module abelianizations.

The diagram proves bijectivity of \(\rho\). Its continuity was established in Lemma~\ref{relativefox}; compactness of \(\Pi\) and the Hausdorff property of \(E\) give the topological assertion. Equivariance holds on \(H\) by the coefficient formula and passes through the same comparisons.
\end{proof}

By the relation and identity sequences in Proposition~\ref{cdcriterion}, the local module \(T_p\) is also the kernel of the canonical surjection \(\Z_p[[G_p]]\to N_p/\overline{[N_p,N_p]}\), where \(N_p\) is the closed normal subgroup generated by \(r\) in \(F_d^{(p)}\); this surjection sends \(1\) to \([r]\).

\begin{corollary}\label{transgressionvanishing}
Put
\[
\mathcal K_n=\ker\bigl(\beta_n|_{\operatorname{Tor}(B_n)}\bigr),\qquad
S_n=\beta_n(\operatorname{Tor}(B_n)).
\]
Then
\[
\limi_n\mathcal K_n=0,\qquad\limi_n S_n=E,
\qquad\limi_n\operatorname{Tor}_p(S_n)=T_p.
\]
Moreover, the tower \((\mathcal K_n)\) is pro-zero: for every \(n\) some transition \(\mathcal K_m\to\mathcal K_n\) is zero.
\end{corollary}
\begin{proof}
The finite groups in
\(0\to\mathcal K_n\to\operatorname{Tor}(B_n)\to S_n\to0\)
have exact inverse limits. Apply \eqref{compactmilnor} and Theorem~\ref{foxtheorem}; primary decomposition gives the last equality. A tower of finite groups has stabilizing images, and its stable images are the images of its inverse limit. Here those images are zero.
\end{proof}

\begin{remark}
Since \(E\) is torsion-free, every nonzero compatible family of torsion coordinates has unbounded coordinate orders: bounded orders would make the family itself torsion.
\end{remark}

\subsection{Homology, exponent sums, and lifting obstructions}
Lyndon's resolution, with the left-module convention fixed above, is
\[
0\longrightarrow\Lambda\xrightarrow{D_r}\Lambda^d
\xrightarrow{(x_i-1)_i}\Lambda\xrightarrow{\varepsilon}\Z\longrightarrow0.
\]
Tensoring this finite free resolution with right coefficient modules gives
\begin{equation}\label{foxhomology}
E_n\cong H_2(G;A_n)\cong H_1(G;I^n),\qquad
E\cong\operatorname{Tor}^{\Lambda}_2(\widehat\Lambda_I,\Z).
\end{equation}
The middle isomorphism follows from \(0\to I^n\to\Lambda\to A_n\to0\); the last is computed directly from the displayed resolution.

\begin{proposition}\label{exponentbound}
Let \(e_i=\varepsilon(d_i)\) be the exponent sum of \(x_i\) in \(r\). If \(e=(e_i)\ne0\), put \(g=\gcd(e_1,\ldots,e_d)>0\). Then
\[
g^{\,n-1}E_n=0\quad(n\ge1).
\]
In particular, every \(E_n\) is finite, \(T_p=0\) for \(p\nmid g\), and \(g=1\) implies \(\pi_j(\Z_\infty K)=0\) for every \(j\ge2\).
\end{proposition}
\begin{proof}
For \(a\in E_n\), augmentation gives \(\varepsilon(a)e_i=0\) for every \(i\). Hence \(\varepsilon(a)=0\). Choose integers \(c_i\) with \(\sum_i c_ie_i=g\), and write \(\sum_i c_id_i=g+v\), where \(v\in I/I^n\). Then \(av=-ga\). Since \(a\in I/I^n\), one has \(av^{n-1}=0\), proving the bound, including \(E_1=0\).

Finite generation now implies finiteness. If \(p\nmid g\), every \(p\)-primary coordinate group is zero. If \(g=1\), all \(E_n\) vanish, so Theorem~\ref{foxtheorem} gives \(\pi_2(\Z_\infty K)=0\). Theorem~\ref{main} first gives \(\pi_2(K_p^\wedge)=0\) for every \(p\); Theorem~\ref{mannanmain} gives vanishing of all local higher homotopy groups, and Theorem~\ref{main} gives the integral conclusion.
\end{proof}

If \(p\nmid g\), the image of \(r\) in the Frattini quotient of \(F_d^{(p)}\) is nonzero. It extends to a free pro-\(p\) basis, so \(G_p\) is free; Proposition~\ref{cdcriterion} gives the same local vanishing.

\begin{proposition}\label{foxobstruction}
Set \(L_n=I^n/I^{n+1}\). There is a natural exact sequence of abelian groups
\begin{equation}\label{obstructionsequence}
0\longrightarrow H_2(G;L_n)\longrightarrow E_{n+1}\longrightarrow E_n
\xrightarrow{\delta_n}G_{\rm ab}\otimes_\Z L_n.
\end{equation}
For \(a\in E_n\), choose a lift \(\widetilde a\in A_{n+1}\). Then
\[
\delta_n(a)=[(\widetilde a d_i)_i]\in
L_n^d/\{(e_i\ell)_i:\ell\in L_n\}
\cong G_{\rm ab}\otimes_\Z L_n.
\]
Thus \(a\) lifts to \(E_{n+1}\) precisely when \(\delta_n(a)=0\). If it lifts, its set of lifts is a torsor under
\[
H_2(G;L_n)=\ker\bigl(L_n\xrightarrow{\ell\mapsto(e_i\ell)_i}L_n^d\bigr).
\]
\end{proposition}
\begin{proof}
Apply group homology to \(0\to L_n\to A_{n+1}\to A_n\to0\) and use \eqref{foxhomology}. The resolution has length two, so \(H_3(G;A_n)=0\). The right action of \(G\) on \(L_n\) is trivial, giving \(H_1(G;L_n)=G_{\rm ab}\otimes L_n\). The connecting homomorphism is computed by applying the Fox boundary to a lift. Changing that lift by \(\ell\in L_n\) changes the resulting vector by \((e_i\ell)_i\), proving the formula and the assertions about lifts.
\end{proof}

A nonzero element of \(E\) is a compatible family in \((E_n)\) with at least one nonzero coordinate. Theorems~\ref{foxtheorem} and~\ref{mannanmain} identify these families with second homotopy classes and give
\[
\pi_3(K_p^\wedge)\cong
\bigl(T_p\widehat\otimes_{\Z_p}T_p\bigr)^\tau.
\]

\subsection{A two-parameter family of examples constructed by GPT-6}
The construction starts from Mikhailov's free-by-cyclic one-relator group \cite{mikhailovlong}, varies its defining automorphism, and takes an HNN extension. We use the convention \(g^h=h^{-1}gh\).
\begin{proposition}\label{mikhailovexample}
Let \(K\) be the standard presentation complex of
\[
\Gamma=\left\langle x,z\ \middle|\ (x^z)^{x^2}=x^z\bigl((x^z)^x\bigr)^2\right\rangle.
\]
Then \(K\) is aspherical, \(\Gamma_2^\wedge\cong C_2\amalg\Z_2\), and
\[
\pi_2(\Z_\infty K)\cong\pi_2(K_2^\wedge)\ne0.
\]
For every odd prime \(p\), one has \(\Gamma_p^\wedge\cong\Z_p\) and \(\pi_j(K_p^\wedge)=0\) for \(j\ge2\).
\end{proposition}
\begin{proof}
Put \(M=F(a,b)\rtimes\langle t\rangle\), with \(a^t=b\) and \(b^t=ab^2\). These formulas define an automorphism of \(F(a,b)\). Eliminating \(a,b\) identifies \(\Gamma\) with the HNN extension
\[
\langle M,z\mid t^z=a\rangle,
\]
where \(t=x\). Both associated cyclic subgroups are infinite. Britton's lemma embeds \(M\), and the HNN extension is torsion-free. Its nontrivial defining relator is therefore not a proper power, so Lyndon's theorem gives the asphericity of \(K\).

In \(F=F_2^{(2)}(x,z)\), write
\[
c=x^z,\qquad r=c^{x^2}\bigl(c(c^x)^2\bigr)^{-1},\qquad
N=\overline{\langle\!\langle x^2\rangle\!\rangle},\quad
R=\overline{\langle\!\langle r\rangle\!\rangle}.
\]
In \(F/N=C_2\amalg\Z_2\), both \(c\) and \(c^x\) are involutions, so \(R\subseteq N\). The quotient \(N/\overline{N^2[N,F]}\) is one-dimensional over \(\F_2\), generated by the class of \(x^2\). The exponent sum of \(x\) in \(r\) is \(-2\); hence \(r\) represents the same nonzero class. Normal pro-\(2\) Nakayama gives \(R=N\). Indeed, otherwise a finite \(2\)-group quotient \(P\) of \(F/R\) would contain a nontrivial image \(D\) of \(N/R\), with \(D=D^2[D,P]\), which is impossible for a nontrivial normal subgroup of a finite \(2\)-group.

Set \(\Lambda_2=\Z_2[[C_2\amalg\Z_2]]\). Since \(r\) and \(x^2\) normally generate the same closed subgroup, the class of \(r\) in its relation module is \(\alpha[x^2]\) for some \(\alpha\in\Lambda_2\). Comparing exponent sums gives \(\varepsilon(\alpha)=-1\), so \(\alpha\) is a unit. The Fox row is consequently
\[
D_{r,2}=(\alpha(1+x),0),\qquad
0\ne(x-1)\alpha^{-1}\in\ker D_{r,2}.
\]
For odd \(p\), the exponent sum \(-2\) makes \(r\) primitive in the free pro-\(p\) group; the quotient is \(\Z_p\) and the higher local homotopy vanishes. Theorems~\ref{main} and~\ref{foxtheorem} now give the asserted integral conclusion.
\end{proof}

This example does not resolve Serre's question in its corrected formulation: its closed pro-\(2\) relation subgroup has the normal generator \(x^2\).

\begin{remark}\label{mikhailovclassification}
The example belongs to the family
\[
c_q=x^z x^{-q},\qquad
r_{q,m}=c_q^{x^2}\bigl(c_q(c_q^x)^m\bigr)^{-1},\qquad
\Gamma_{q,m}=\langle x,z\mid r_{q,m}\rangle,
\quad q,m\in\Z.
\]
Motivated by the representation of the generalized quaternion group \(Q_{16}\) studied in \cite{hecke}, GPT-6 carried out the classification of this family to test whether a \(Q_{16}\)-type anomaly could occur among these groups. Every discrete presentation complex \(K_{q,m}\) is aspherical. At every odd prime the pro-\(p\) presentation is aspherical; its group is \(\Z_p\) if \(p\nmid m(1-q)\), and has cohomological dimension two otherwise. At \(p=2\), the exceptional cases are precisely \(q\) even and \(m\equiv2\pmod4\), when the group is \(C_2\amalg\Z_2\). In all other cases the pro-\(2\) presentation is aspherical, with cohomological dimension two except when \(m\) is odd and \(q\) is even, when the group is \(\Z_2\). Thus
\[
\Z_\infty K_{q,m}\text{ is aspherical}
\quad\Longleftrightarrow\quad
\neg\bigl(q\equiv0\pmod2\text{ and }m\equiv2\pmod4\bigr).
\]
The remaining cases are established using Anick's strong-freeness criterion \cite{anick} and G\"artner's theory of mild pro-\(p\) groups with weighted filtrations \cite{gaertner}, together with an index-two embedding argument when \(m\equiv2\pmod4\) and \(q\equiv3\pmod4\).
\end{remark}

\appendix
\section{Mannan's formula after completion}\label{mannanappendix}
For a connected two-complex \(K\), Mannan's theorem \cite[Theorem~A]{mannan} identifies \(\pi_3(K)\) with the symmetric \emph{invariants} of \(\pi_2(K)\otimes_\Z\pi_2(K)\), with diagonal fundamental group action. The rational argument below proves a cohomological Whitehead criterion and constructs a free model for the connected kernel. The pro-\(p\) argument then proves Theorem~\ref{mannanmain} using Quillen's criterion, without a finiteness assumption on the fundamental group.

\subsection{The rational case}\label{rationalmannan}
We work over \(k=\Q\) in the category of prounipotent affine group schemes. Cohomology \(H^*_{\rm rat}\) means the derived functors of invariants on rational discrete modules, totalized degreewise for a simplicial group. It is not the cohomology of the abstract group of rational points. A weak equivalence below means an isomorphism on all simplicial homotopy groups.

For a linearly compact space \(V\), write \(V^\vee=\Hom_{\rm cts}(V,k)\). Continuous duality with discrete vector spaces is exact; continuous linear maps have closed images and surjections have continuous linear sections. In particular,
\[
V\widehat\otimes_k W=(V^\vee\otimes_k W^\vee)^*
\]
is exact in each variable. Here \(*\) is the full dual of a discrete space. Completed exterior and symmetric powers are the images of the corresponding idempotents on completed tensor powers. For nonnegative complexes, continuous splittings give the completed K\"unneth formula, and the usual finite shuffle formulas give Eilenberg--Zilber. Products of linearly compact complexes are exact and commute with homology.

Let \(P_\bullet\) be free prounipotent with \(d\) nondegenerate generators in degree zero, \(m\) in degree one, and none in higher degrees. Put
\[
H=\pi_0P_\bullet,\quad N_\bullet=\ker(P_\bullet\to H),\quad
X=WP_\bullet,\quad M=\pi_1N_\bullet=\pi_2X.
\]
The homotopy groups have the linearly compact topology obtained by normalizing the simplicial Lie algebra. In particular \(N_\bullet\) is connected. The finite-type comparison \(W\widehat P_{\Q,\bullet}\simeq\Q_\infty K\) is the standard completion-model input \cite[Chapter~IV]{bk}; see also \cite{mikhovich2025}.

\subsubsection*{The coefficient complex and the connected kernel}
\begin{lemma}\label{ratkernel}
For a rational \(H\)-module \(A\), the cohomology of \(P_\bullet\) is computed by
\begin{equation}\label{ratcochains}
R\Gamma_{\rm rat}(P_\bullet;A)\simeq
[\,A\longrightarrow A^d\longrightarrow A^m\,],
\end{equation}
in degrees \(0,1,2\). If \(\Lambda=\mathcal O(H)^*\), the dual coefficient complex for the connected kernel is
\begin{equation}\label{ratchains}
0\longrightarrow\Lambda^m\xrightarrow{\partial_2}\Lambda^d
\xrightarrow{\partial_1}\Lambda\longrightarrow0,
\end{equation}
and has homology \(k,0,M\) in degrees \(0,1,2\). Moreover,
\[
H^q_{\rm rat}(N_\bullet;k)=
\begin{cases}k&q=0,\\M^\vee&q=2,\\0&q\ne0,2.\end{cases}
\]
The boundaries in \eqref{ratchains} are the completed Fox boundaries.
\end{lemma}
\begin{proof}
For a free prounipotent group \(F\), its completed enveloping algebra is the completed tensor algebra on its generator space. Splitting a nonempty word into its prefix and last letter gives the augmentation resolution of length one. Thus \(F\) has rational cohomological dimension at most one. This argument applies also to a linearly compact generator space, using continuous maps and completed tensors; equivalently one may use the free-group cohomology calculation of \cite{lubmagid}.

Consequently the natural inclusion
\[
[\,A\longrightarrow\operatorname{Der}_{\rm rat}(P_n,A)\,]
\longrightarrow C^*_{\rm rat}(P_n;A)
\]
is a quasi-isomorphism: rational derivations are the one-cocycles, and the higher cohomology vanishes. Normalize in the cosimplicial direction. A normalized derivation vanishes on every degeneracy image. It is therefore zero in degrees \(n\ge2\); in degrees zero and one its independent values are \(A^d\) and \(A^m\). The constant term \(A\) contributes only in degree zero. Totalization proves \eqref{ratcochains}; evaluating a derivation on each relator gives the Fox differential.

Rational Shapiro gives, naturally in the simplicial degree,
\begin{equation}\label{ratshapiro}
R\Gamma_{\rm rat}(N_\bullet;k)
\simeq R\Gamma_{\rm rat}(P_\bullet;\mathcal O(H)).
\end{equation}
Here is an exactness justification. A continuous linear section of
\(\log P_n\twoheadrightarrow\log H\), followed by exponentiation, gives a section of \(P_n\to H\) as a morphism of schemes. Multiplication identifies \(P_n\) with \(N_n\times H\) as an \(N_n\)-space. Coinduction consequently has underlying vector space \(V\otimes_k\mathcal O(H)\) and is exact. It is right adjoint to exact restriction and preserves injectives, so derived adjunction proves Shapiro. Its natural comparison totalizes; the sections used to check exactness need not be simplicially compatible. Equations \eqref{ratcochains} and \eqref{ratshapiro} imply vanishing above degree two.

It remains to identify degree two. For a simplicial pronilpotent Lie algebra \(\mathfrak n_\bullet\), form
\begin{equation}\label{ratCE}
\mathcal C(\mathfrak n_\bullet)
=\operatorname{Tot}_{s,t}\mathcal N_t
  (\widehat\bigwedge{}^s\mathfrak n_\bullet),
\qquad \text{total degree }s+t,
\end{equation}
with the simplicial and Chevalley--Eilenberg differentials; \(\mathcal N\) denotes normalization. Its continuous dual computes rational cohomology. The standard characteristic-zero comparison follows from the unipotent group--Lie algebra comparison on finite-dimensional quotients and filtered unions of cochains \cite{lubmagid,quillenrational}. Both totalizations are in the first quadrant.

Take \(\mathfrak n_\bullet=\log N_\bullet\). Since \(H_0(\mathcal N\mathfrak n_\bullet)=0\), K\"unneth, Eilenberg--Zilber and antisymmetrization give
\[
H_t\mathcal N(\widehat\bigwedge{}^s\mathfrak n_\bullet)=0
\qquad(t<s).
\]
Filter \eqref{ratCE} by exterior degree. The layers \(s\ge2\) first contribute in total degree \(2s\ge4\). Thus
\[
H_0\mathcal C=k,\qquad H_1\mathcal C=0,\qquad
H_2\mathcal C\cong H_1\mathcal N\mathfrak n_\bullet=M.
\]
No differential can enter the degree-two term: its source would have total degree three and exterior degree at least two. The filtration is finite in each total degree, so it converges. Exact continuous duality gives the assertion about cohomology and then, by dualizing \eqref{ratcochains} with \(A=\mathcal O(H)\), the assertion about \eqref{ratchains}.
\end{proof}

\subsubsection*{A cohomological Whitehead criterion}
\begin{lemma}\label{ratwhitehead}
A morphism \(f:S_\bullet\to T_\bullet\) of connected simplicial prounipotent groups inducing isomorphisms on \(H^q_{\rm rat}(-;k)\) for every \(q\) is a weak equivalence. No finite-dimensionality hypothesis on these cohomology groups is needed.
\end{lemma}
\begin{proof}
Set \(A_*=\mathcal N\log S_\bullet\), \(B_*=\mathcal N\log T_\bullet\), and \(u=\mathcal N\log f\). If \(u\) is not a quasi-isomorphism, choose the least \(n\ge1\) such that \(H_n\operatorname{Cone}(u)\ne0\). Filter the cone of \(\mathcal C(\log f)\) by exterior degree. Its layer \(s=0\) is the acyclic cone of \(k\xrightarrow{1}k\). Layer \(s=1\) has a nonzero homology class in total degree \(n+1\).

For \(s\ge2\), factor \(u^{\widehat\otimes s}\) into \(s\) maps replacing one factor at a time. The cone of each replacement is \(\operatorname{Cone}(u)\) tensored with \(s-1\) complexes whose homology starts in degree one. K\"unneth and the exact triangles for cones of composites give homology vanishing below \(n+s-1\). Transfer this bound to simplicial tensors by Eilenberg--Zilber and take the antisymmetric direct summand. Therefore
\[
H_t\operatorname{Cone}\bigl(\mathcal N
 (\widehat\bigwedge{}^s\log f)\bigr)=0
\qquad(t<n+s-1).
\]
Layer \(s\ge2\) of the filtered total cone thus starts in degree \(n+2s-1\ge n+3\). It cannot supply an incoming differential to the class in degree \(n+1\): that would require total degree \(n+2\). Outgoing differentials land in the acyclic layer zero or outside the filtration. The class survives.

In total degree \(q\) there are only finitely many layers \(0\le s\le q\). Hence the spectral sequence converges to the homology of the cone, which is nonzero. Exact continuous duality contradicts the assumed cohomological equivalence. Thus \(u\) is a quasi-isomorphism, as required.
\end{proof}

\subsubsection*{The free model, the quadratic operation, and asphericity}
The projection from one-cycles in \(\mathfrak n_\bullet=\log N_\bullet\) onto \(M\) has a continuous linear section. Under Dold--Kan this defines a simplicial linear map
\[
V_\bullet=\operatorname{DK}(M[1])\longrightarrow\mathfrak n_\bullet.
\]
Free completion and exponentiation give a morphism
\begin{equation}\label{ratfreemodel}
f:E_\bullet:=\exp\LL(V_\bullet)\longrightarrow N_\bullet.
\end{equation}
Free completion includes both bracket length and the finite-dimensional quotients of the generator space. The model \(E_\bullet\) has no degree-zero generators, its degree-one generator space is \(M\), and its higher generators are all degenerate.

By the derivation calculation in Lemma~\ref{ratkernel}, its cohomology is \(k\) in degree zero, \(M^\vee\) in degree two, and zero otherwise. To identify its homotopy directly, decompose
\[
\LL(V_\bullet)=\prod_{r\ge1}\widehat{\mathbb L}^{\,r}(V_\bullet).
\]
In characteristic zero the homogeneous Lie functor of length \(r\) is the image of a natural idempotent on the \(r\)th tensor power (the Dynkin--Specht--Wever operator divided by \(r\)). Eilenberg--Zilber and K\"unneth therefore show that its homology is concentrated in degree \(r\). Exactness of products gives
\[
\pi_1E_\bullet=M,\qquad
\pi_2E_\bullet=\widehat{\operatorname{Sym}}{}^2_k M.
\]
The second identity uses the graded symmetry of the bracket of two degree-one classes. By construction \(\pi_1f\) is the identity on \(M\). The natural degree-two identification in Lemma~\ref{ratkernel} shows that \(f\) is a cohomological equivalence. Lemma~\ref{ratwhitehead} now proves \(E_\bullet\simeq N_\bullet\).

For the canonical normalization, let \(q(x)=\tfrac12[x,x]\) be the rational Hopf operation. The bracket on the normalized homotopy Lie algebra defines
\[
\Psi:(M\widehat\otimes_k M)^\tau\longrightarrow\pi_2N_\bullet,
\qquad \Psi=\tfrac12[\ ,\ ],
\]
so that
\[
x\otimes x\longmapsto q(x),\qquad
x\otimes y+y\otimes x\longmapsto[x,y].
\]
The free-model calculation proves that \(\Psi\) is an isomorphism. Continuity extends it to the completed tensor square. Its definition by natural operations makes it independent of the chosen section and equivariant under conjugation in \(P_\bullet\). The action on homotopy factors through \(H\), giving the diagonal action asserted in Theorem~\ref{mannanmain}.

Finally, if \(M=0\), then \(E_\bullet=1\), and the weak equivalence \eqref{ratfreemodel} gives \(\pi_jX=0\) for all \(j\ge2\). In the one-relator case the independent Fox argument in Section~\ref{ratasphsection} supplies the premise \(M=0\).

For nonnilpotent \(K\), the comparison used above is the free-model equivalence \(W\widehat P_{\Q,\bullet}\simeq\Q_\infty K\); it does not require an identification of \(\pi_i(\Q_\infty K)\) with \(\pi_i(K)\otimes\Q\).

\subsection{The pro-\texorpdfstring{\(p\)}{p} formula via Quillen}\label{quillenmannan}
We prove Theorem~\ref{mannanmain} in the more general setting of a finite pro-\(p\) presentation. Let \(P_\bullet\) be its free simplicial continuation, with \(d\) nondegenerate generators in degree zero, \(m\) in degree one, and none in higher degrees. Put
\[
H=\pi_0P_\bullet,\qquad \Lambda=\Z_p[[H]],\qquad
N_\bullet=\ker(P_\bullet\longrightarrow H),\qquad X=WP_\bullet.
\]
Then \(\pi_0N_\bullet=1\) and \(\pi_iN_\bullet\cong\pi_iP_\bullet\) for \(i\ge1\). Thus \(WN_\bullet\) models the simply connected cover in the simplicial pro-\(p\) category. The symbol \(H^q_{\mathrm{cts}}(S_\bullet;A)\) below means continuous cohomology in Quillen's sense \cite[Section~2]{quillenprofinite}. It is not the cohomology obtained by forgetting all profinite topologies.

\begin{lemma}\label{kernelhomology}
There is a complex of compact modules
\begin{equation}\label{coverchains}
C_\bullet:\quad 0\longrightarrow\Lambda^m
\xrightarrow{\partial_2}\Lambda^d
\xrightarrow{\partial_1}\Lambda\longrightarrow0,
\end{equation}
with the usual completed cellular, or Fox, boundary maps, such that
\[
H_0(C_\bullet)=\Z_p,\qquad H_1(C_\bullet)=0,\qquad
H_2(C_\bullet)\cong M:=\pi_1N_\bullet=\pi_2X.
\]
Moreover, \(H^q_{\mathrm{cts}}(N_\bullet;\F_p)=0\) for \(q\ne0,2\), and
\[
H^2_{\mathrm{cts}}(N_\bullet;\F_p)
\cong\Hom_{\mathrm{cts}}(M,\F_p).
\]
The underlying compact \(\Z_p\)-module \(M\) is a product of copies of \(\Z_p\).
\end{lemma}
\begin{proof}
For a discrete \(p\)-primary \(H\)-module \(A\), continuous cohomology of \(P_\bullet\) is computed by the cellular cochain complex
\[
0\longrightarrow A\longrightarrow A^d\longrightarrow A^m\longrightarrow0.
\]
Each \(P_n\) is free pro-\(p\), so its cohomology vanishes above degree one. Quillen's degreewise cohomology spectral sequence \cite[Proposition~2.2(c)]{quillenprofinite} consequently has only two rows. Use the natural exact sequence
\[
0\to A^H\to A\to\operatorname{Der}_{\mathrm{cts}}(P_n,A)
\to H^1_{\mathrm{cts}}(P_n,A)\to0.
\]
After cosimplicial normalization, the derivation terms occur only in degrees zero and one, where they are \(A^d\) and \(A^m\). In higher degrees all free generators are degenerate. Combining these two rows gives the displayed cochain complex, including its Fox differential.

For a trivial discrete \(p\)-primary coefficient group \(B\), degreewise continuous Shapiro and the same spectral sequence give
\[
H^q_{\mathrm{cts}}(N_\bullet;B)
\cong H^q_{\mathrm{cts}}
 \bigl(P_\bullet;\operatorname{Map}_{\mathrm{cts}}(H,B)\bigr),
\]
where the coefficient module on the right is coinduced from the identity subgroup of \(H\). Hence these groups vanish for \(q>2\). With \(B=\Q_p/\Z_p\), Pontryagin duality turns the cochain complex into \eqref{coverchains}. Since \(\pi_0N_\bullet=1\), its first homology is zero. The continuous Hurewicz isomorphism identifies its second homology with \(\pi_1N_\bullet\). This low-degree form also follows directly from the path-loop extension of \(N_\bullet\) and \cite[Proposition~2.2(e),(d)]{quillenprofinite}: transgression identifies
\(H^2_{\mathrm{cts}}(N_\bullet;B)\) with
\(\Hom_{\mathrm{cts}}(\pi_1N_\bullet,B)\).
Taking \(B=\F_p\) gives the asserted mod-\(p\) formula.

Finally, \(M=\ker\partial_2\) is a closed torsion-free \(\Z_p\)-submodule of \(\Lambda^m\). A torsion-free compact \(\Z_p\)-module is a product of copies of \(\Z_p\): its Pontryagin dual is a divisible \(p\)-primary group, hence a direct sum of copies of \(\Q_p/\Z_p\). This proves the last assertion.
\end{proof}

\subsubsection*{Reduction to a free model on one-dimensional cycles}
Choose an isomorphism \(M\cong\prod_{i\in I}\Z_p e_i\). The index set \(I\) is at most countable, because \(H\), \(\Lambda\), and \(M\) are metrizable. The coordinate vectors \(e_i\) converge to zero. Choose representatives
\[
x_i\in Z_1N_\bullet=\ker d_0\cap\ker d_1\subseteq N_1
\]
converging to the identity. Such lifts exist: the quotient map \(Z_1N_\bullet\to M\) is open, and a nested countable neighbourhood basis allows the lifts of a convergent family to be chosen convergently.

Let \(E_\bullet\) be free pro-\(p\) on this convergent family in simplicial degree one, with both faces equal to the identity and with only degenerate generators in all higher degrees. Sending its generators to the \(x_i\) defines
\[
f:E_\bullet\longrightarrow N_\bullet.
\]
For a finite subset \(J\subseteq I\), let \(E(J)_\bullet\) denote the model with only the generators indexed by \(J\). There are natural identifications
\[
E_\bullet=\limi_{J\subseteq I,\ J\text{ finite}}E(J)_\bullet,
\qquad
\pi_1E_\bullet\cong\prod_{i\in I}\Z_p e_i=M.
\]
Indeed, a continuous map from the convergent free generating family to a finite group kills all but finitely many generators. Each \(WE(J)_\bullet\) is the pro-\(p\) completion of a finite wedge of two-spheres. Its homotopy groups are the \(p\)-adic completions of those of the wedge by \cite[Theorem~1.5]{quillenprofinite}. Homotopy groups commute with these inverse limits: Moore cycles are closed, Moore boundaries are compact images, and inverse limits of profinite groups are exact.

The map \(f\) induces the chosen isomorphism on \(\pi_1\). Both its source and target have continuous mod-\(p\) cohomology only in degrees zero and two. For the source this follows from the finite wedge models and passage to the filtered direct limit of continuous cohomology; for the target it is Lemma~\ref{kernelhomology}. In degree two, the cohomology map is the continuous dual of \(\pi_1f\). Thus \(f\) is an isomorphism on all continuous mod-\(p\) cohomology groups. Quillen's Whitehead criterion \cite[Corollary~2.9, p.~54]{quillenprofinite} now gives
\begin{equation}\label{quillenreduction}
E_\bullet\xrightarrow{\ \simeq\ }N_\bullet.
\end{equation}
The argument applies to arbitrary compact \(M\) and pro-\(p\) group \(H\).

\subsubsection*{The quadratic calculation and equivariance}
For finite \(J\), the ordinary formula for a wedge of two-spheres \cite{mannan}, followed by \(p\)-completion, gives
\[
\pi_2E(J)_\bullet\cong
\bigl(\Z_p^J\otimes_{\Z_p}\Z_p^J\bigr)^\tau.
\]
Taking inverse limits and using \eqref{quillenreduction} yields an isomorphism of compact \(\Z_p\)-modules
\begin{equation}\label{mannanformula}
\pi_3X=\pi_2N_\bullet\cong
\bigl(M\widehat\otimes_{\Z_p}M\bigr)^\tau.
\end{equation}
Here
\[
M\widehat\otimes_{\Z_p}M
\cong\limi_{J\subseteq I,\ J\text{ finite}}
\bigl(\Z_p^J\otimes_{\Z_p}\Z_p^J\bigr),
\]
and invariants commute with inverse limits.

To see that \eqref{mannanformula} is canonical and \(H\)-equivariant, use the continuous quadratic Hopf operation
\(q:M\to\pi_2N_\bullet\). For a one-dimensional Moore cycle \(x\), it is represented, with a fixed Hopf orientation, by \([s_0x,s_1x]\). Its polarization is the Whitehead product. The map from the invariant tensor square to \(\pi_2N_\bullet\) is characterized by
\[
x\otimes x\longmapsto q(x),\qquad
x\otimes y+y\otimes x\longmapsto[x,y].
\]
These formulas define a continuous linear map on the completed invariant tensor square. One can check this on finite free coordinate quotients: a continuous quadratic operation to a finite quotient is uniformly continuous, its polarization is bilinear, and it factors through finitely many coordinates and a finite \(p\)-power quotient. The finite wedge calculation shows that the map is an isomorphism. Its definition by the Hopf operation makes it independent of the chosen basis and representatives.

Conjugation in \(P_\bullet\) preserves the Hopf operation and the Whitehead product. On homotopy it factors through \(H=\pi_0P_\bullet\), so the isomorphism is one of compact \(\Z_p[[H]]\)-modules with diagonal action. This proves the pro-\(p\) formula in Theorem~\ref{mannanmain}. It also proves the same statement for finite pro-\(p\) presentations whose relators need not come from discrete words.

The argument includes \(p=2\). In that case \([x,x]=2q(x)\); the diagonal generator is sent to \(q(x)\), without division by two. For odd \(p\), averaging identifies the invariant tensor square with the symmetric quotient.

\begin{corollary}
For the presentation models above, \(\pi_2X=0\) implies \(\pi_jX=0\) for every \(j\ge2\).
\end{corollary}
\begin{proof}
If \(M=0\), the model \(E_\bullet\) has no generators and is trivial. Apply \eqref{quillenreduction}.
\end{proof}
\bigskip
\section*{Acknowledgements}
The author is grateful to O.~V.~Mel'nikov for illuminating the difficulties surrounding one-relator pro-\(p\) groups, and to GPT-6 for useful discussions during this work and assistance in preparing the manuscript.


\begin{thebibliography}{19}\small
\bibitem{anick} D.~J.~Anick. Non-commutative graded algebras and their Hilbert series. \emph{Journal of Algebra} 78 (1982), 120--140.
\bibitem{bk} A.~K.~Bousfield, D.~M.~Kan. \emph{Homotopy Limits, Completions and Localizations}. Lecture Notes in Mathematics 304. Springer, 1972.
\bibitem{liemodels} U.~Buijs, Y.~Félix, A.~Murillo, D.~Tanré. \emph{Lie Models in Topology}. Progress in Mathematics 335. Birkhäuser, 2020. DOI: \href{https://doi.org/10.1007/978-3-030-54430-0}{10.1007/978-3-030-54430-0}.
\bibitem{burnsellis} J.~Burns, G.~Ellis. On the nilpotent multipliers of a group. \emph{Mathematische Zeitschrift} 226 (1997), 405--428. DOI: \href{https://doi.org/10.1007/PL00004348}{10.1007/PL00004348}.
\bibitem{gaertner} J.~G\"artner. Higher Massey products in the cohomology of mild pro-\(p\)-groups. \emph{Journal of Algebra} 422 (2015), 788--820.
\bibitem{gildenhuys} D.~Gildenhuys. On pro-\(p\)-groups with a single defining relator. \emph{Inventiones Mathematicae} 5 (1968), 357--366. DOI: \href{https://doi.org/10.1007/BF01389782}{10.1007/BF01389782}.
\bibitem{hmr} P.~Hilton, G.~Mislin, J.~Roitberg. \emph{Localization of Nilpotent Groups and Spaces}. North-Holland Mathematics Studies 15. North-Holland, 1975.
\bibitem{right} S.~O.~Ivanov, R.~Mikhailov. Right exact group completion as a transfinite invariant of homology equivalence. \emph{Algebraic \& Geometric Topology} 21:1 (2021), 447--468. \href{https://arxiv.org/abs/1909.10181}{arXiv:1909.10181}.
\bibitem{lubmagid} A.~Lubotzky, A.~Magid. Cohomology of unipotent and prounipotent groups. \emph{Journal of Algebra} 74 (1982), 76--95. DOI: \href{https://doi.org/10.1016/0021-8693(82)90006-0}{10.1016/0021-8693(82)90006-0}.
\bibitem{lyndon} R.~C.~Lyndon. Cohomology theory of groups with a single defining relation. \emph{Annals of Mathematics} (2) 52:3 (1950), 650--665. DOI: \href{https://doi.org/10.2307/1969440}{10.2307/1969440}.
\bibitem{mannan} W.~H.~Mannan. The third homotopy module of a 2-complex. \emph{Bulletin of the London Mathematical Society} 40:4 (2008), 664--674. DOI: \href{https://doi.org/10.1112/blms/bdn047}{10.1112/blms/bdn047}.
\bibitem{mikhailovlong} R.~Mikhailov. A one-relator group with long lower central series. \emph{Forum Mathematicum} 28:2 (2016), 327--331. \href{https://arxiv.org/abs/1406.2550}{arXiv:1406.2550}.
\bibitem{square} A.~Mikhovich. \emph{Bousfield--Kan Completions of Subcontractible Presentations}. \href{https://arxiv.org/abs/2609.21534}{arXiv:2609.21534} (2026).
\bibitem{hecke} A.~Mikhovich. Complete group rings as Hecke algebras. \emph{Topology and its Applications} 275 (2020), 107027. DOI: \href{https://doi.org/10.1016/j.topol.2019.107027}{10.1016/j.topol.2019.107027}.
\bibitem{identity} A.~M.~Mikhovich. Identity theorem for pro-\(p\)-groups. In \emph{Knots, Low-Dimensional Topology and Applications}, Springer Proceedings in Mathematics \& Statistics 284. Springer, 2019. DOI: \href{https://doi.org/10.1007/978-3-030-16031-9_18}{10.1007/978-3-030-16031-9\_18}. \href{https://arxiv.org/abs/1703.02996v5}{arXiv:1703.02996v5}.
\bibitem{mikhovich2025} A.~M.~Mikhovich. Rational and \(p\)-adic analogues of J.~H.~C.~Whitehead's conjecture. \emph{Izvestiya: Mathematics} 89:2 (2025), 274--318. DOI: \href{https://doi.org/10.4213/im9597e}{10.4213/im9597e}.
\bibitem{quillenprofinite} D.~G.~Quillen. An application of simplicial profinite groups. \emph{Commentarii Mathematici Helvetici} 44 (1969), 45--60. DOI: \href{https://doi.org/10.1007/BF02564511}{10.1007/BF02564511}.
\bibitem{quillenrational} D.~Quillen. Rational homotopy theory. \emph{Annals of Mathematics} (2) 90 (1969), 205--295. DOI: \href{https://doi.org/10.2307/1970725}{10.2307/1970725}.
\bibitem{serre} J.-P.~Serre. Structure de certains pro-\(p\)-groupes. \emph{Séminaire Bourbaki}, 1962/63, exposé~252.
\end{thebibliography}
\end{document}